\documentclass[letterpaper,10pt,conference]{ieeeconf}  
\IEEEoverridecommandlockouts 
\usepackage{amsmath,mathptmx}
\usepackage[pdftex]{graphicx}
\usepackage[caption=false,font=footnotesize]{subfig}
  \DeclareGraphicsExtensions{.pdf}
\usepackage{algorithm}
\newcommand{\eps}{\varepsilon}
\newcommand{\supCN}{\mbox{$\sup {\rm CN}$}}
\newcommand{\supcCN}{\mbox{$\sup {\rm cCN}$}}
\newcommand{\suptwocCN}{\mbox{$\sup {\rm mcCN}$}}
\newcommand{\supC}{\mbox{$\sup {\rm C}$}}

\newtheorem{theorem}{Theorem}
\newtheorem{problem}[theorem]{Problem}
\newtheorem{lem}[theorem]{Lemma}

\newtheorem{definition}[theorem]{Definition}

\title{\LARGE \bf Combined Top-down and Bottom-up Approach\\ to Multilevel Supervisory Control}
\author{Jan Komenda, Tom{\' a}{\v s} Masopust, and Jan H. van Schuppen
  \thanks{J. Komenda is with the Institute of Mathematics,
          Academy of Sciences of the Czech Republic,
          {\v Z}i{\v z}kova 22, 616 62 Brno, Czech Republic.
          T. Masopust is with TU Dresden, Germany, and with the Institute of Mathematics, Academy of Sciences of the Czech Republic.
          J. H. van Schuppen is with
          Van Schuppen Control Research,
          Gouden Leeuw 143, 1103 KB, Amsterdam, The Netherlands.
          {\tt\small komenda@math.cas.cz, masopust@ipm.cz, jan.h.van.schuppen@xs4all.nl}
  }%
}
\begin{document}

\maketitle
\thispagestyle{empty}
\pagestyle{empty} 

\begin{abstract}
  Recently, we have proposed two complementary approaches, top-down and bottom-up, to multilevel supervisory control of discrete-event systems. In this paper, we compare and combine these approaches. The combined approach has strong  features of both approaches, namely, a lower complexity of the top-down approach with the generality of the bottom-up approach. We show that, for prefix-closed languages, a posteriori supervisors computed in the bottom-up manner do not alter maximal permissiveness within the three-level coordination control architecture, that is, the supremal three-level conditionally-controllable and conditionally-normal language can always be computed in a distributed way using multilevel coordination. Moreover, a general polynomial-time procedure for non-prefix closed case is proposed  based on coordinators for nonblockingness and a posteriori supervisors.
\end{abstract}

\section{Introduction}

  Discrete-event abstractions of complex engineering systems have often a modular structure and
  typically consist of either a large Petri nets or a network (synchronous product) of finite automata. 
  Supervisory control theory was introduced to provide a formal guarantee of
  safety and nonblockingness for these systems. Modular and decentralized supervisory control
  theories are especially relevant for large scale systems and these are often combined with
  hierarchical control based on abstractions. 
  Coordination control of distributed systems with synchronous communication was
  developed by the authors, 
  see~\cite{JDEDS} and the references therein,
  in which a coordinator restricts 
  the behavior of two or more subsystems 
  so that, after further control synthesis,
  safety and nonblockingness of the distributed system are achieved.
  
  In order to further decrease the complexity of control synthesis,
  a multilevel coordination control framework was proposed in~\cite{CDC14}, where a single (central) coordinator at the top level
  of the standard (three-level) coordination control was replaced
  by group supervisors for different group systems at the lowest level.
  These coordinators together with their supervisors then form the middle
  (intermediate) level, while a (single) high-level coordinator is
  at the top level of the three-level coordination control.
  This architecture considerably limits the computational complexity 
  due to relatively small event sets at the various levels.
 
  Recently, we proposed two complementary approaches, called top-down \cite{CDC14} and bottom-up \cite{ECC14},  to multilevel supervisory control of  discrete-event systems. 
  We have developed constructive results in the top-down approach of~\cite{CDC14}, 
  where it was shown under which conditions
  the maximally permissive solution for the three-level coordination control architecture exists,
  that is, the supremal three-level conditionally controllable languages.

  In this paper, we propose a combined approach, which can be described as a top-down design followed by
  a bottom-up computation. The combined approach combines the strong features,
  namely, the lower complexity of the top-down approach
  with the generality of the bottom-up approach. More specifically,
  we propose to complete the top-down design of coordinators from the high-level
  to the bottom-level by computing a posteriori supervisors on these coordinator alphabets
  in the opposite direction, i.e., in the bottom-up manner.
  The role of these supervisors is to enforce the sufficient conditions
  for distributed computation presented in~\cite{ACC15},
  which are formulated as controllability and normality on all coordinator alphabets.
  Note that unlike the bottom-up approach of~\cite{ECC14}, we do not
  need to compute supervisors at the higher level, but only supervisors for 
  individual subsystems at the lowest level are computed.

  Moreover, we show that for prefix-closed languages a posteriori supervisors do not
  alter maximal permissiveness within the three-level coordination control architecture, i.e.,
  the supremal three-level conditionally-controllable and
  conditionally-normal languages can always be computed in
  the distributed way. In the general case of non-prefix-closed languages,
  we propose to compute coordinators for nonblockingness in the bottom-up manner
  in addition to the a posteriori supervisors.

  This paper has the following structure.
  In Section~\ref{sec:prel}, we recall the basic elements of
  supervisory control theory together with 
  basic (three-level) coordination control framework.
  In Section~\ref{sec:sup2cc}, multilevel coordination control framework is 
  discussed and the strong points and drawbacks of the two existing approaches are compared.
  The main results of the paper are presented in Sections~\ref{sec:combined} and~\ref{sec:general}. 
  In the former section, it is proven that in the combined approach based on a posteriori supervisors, the supremal three-level conditionally-controllable and conditionally-normal languages can always be computed in a distributed way. Then, in the latter section concerned with the non-prefix-closed case, a formal general procedure is presented, where a posteriori supervisors are combined with coordinators for nonblockingness.

\section{Preliminaries}
\label{sec:prel}

This section recalls the basic results about coordination control of  partially observed DES
with a single (centralized) coordinator.
First, elementary notions and notation  of supervisory control theory are recalled. The reader is referred to~\cite{CL08} for more details. 

Let $A$ be a finite nonempty set of {\em events}, and let $A^*$ denote the set of all finite words over $A$. The {\em empty word\/} is denoted by $\eps$. 

  A {\em generator\/} is a quintuple $G=(Q,A, f, q_0, Q_m)$, where $Q$ is the finite nonempty set of {\em states}, $A$ is the {\em event set\/}, $f: Q \times A \to Q$ is the {\em partial transition function}, $q_0 \in Q$ is the {\em initial state}, and $Q_m\subseteq Q$ is the set of {\em marked states}. In the usual way, the transition function $f$ can be extended to the domain $Q \times A^*$ by induction. The behavior of $G$ is described in terms of languages. The language {\em generated\/} by $G$ is the set $L(G) = \{s\in A^* \mid f(q_0,s)\in Q\}$ and the language {\em marked\/} by $G$ is the set $L_m(G) = \{s\in A^* \mid f(q_0,s)\in Q_m\}\subseteq L(G)$.

  A {\em (regular) language\/} $L$ over an event set $A$ is a set $L\subseteq A^*$ such that there exists a generator $G$ with $L_m(G)=L$. The prefix closure of $L$ is the set $\overline{L}=\{w\in A^* \mid \text{there exists } u \in A^* \text{ such that } wu\in L\}$; $L$ is {\em prefix-closed\/} if $L=\overline{L}$. 

  A {\em (natural) projection} $P: A^* \to A_o^*$, for some $A_o\subseteq A$, is a homomorphism defined so that $P(a)=\eps$, for $a\in A\setminus A_o$, and $P(a)=a$, for $a\in A_o$. The {\em inverse image} of $P$, denoted by $P^{-1} : A_o^* \to 2^{A^*}$, is defined as $P^{-1}(s)=\{w\in A^* \mid P(w) = s\}$. The definitions can naturally be extended to languages. The projection of a generator $G$ is a generator $P(G)$ whose behavior satisfies $L(P(G))=P(L(G))$ and $L_m(P(G))=P(L_m(G))$.

  A {\em controlled generator with partial observations\/} is a structure $(G,A_c,P,\Gamma)$, where $G$ is a generator over $A$, $A_c \subseteq A$ is the set of {\em controllable events}, $A_{u} = A \setminus A_c$ is the set of {\em uncontrollable events}, $P:A^*\to A_o^*$ is the projection, and $\Gamma = \{\gamma \subseteq A \mid A_{u}\subseteq\gamma\}$ is the {\em set of control patterns}. 
  
  A {\em supervisor\/} for the controlled generator $(G,A_c,P,\Gamma)$ is a map $S:P(L(G)) \to \Gamma$. 
  
  A {\em closed-loop system\/} associated with the controlled generator $(G,A_c,P,\Gamma)$ and the supervisor $S$ is defined as the smallest language $L(S/G) \subseteq A^*$ such that 
  \begin{enumerate}
    \item $\eps \in L(S/G)$ and 
    \item if $s \in L(S/G)$, $sa\in L(G)$, and $a \in S(P(s))$, then also $sa \in L(S/G)$.
  \end{enumerate}
  The marked behavior of the closed-loop system is defined as $L_m(S/G)=L(S/G)\cap L_m(G)$.

  Let $G$ be a generator over $A$, and let $K\subseteq L_m(G)$ be a specification.  The aim of supervisory control theory is to find a nonblocking 
  supervisor $S$ such that $L_m(S/G)=K$. 
   The nonblockingness means that $\overline{L_m(S/G)} = L(S/G)$, hence $L(S/G)=\overline{K}$. 
  It is known that such a supervisor exists if and only if $K$ is 
  \begin{enumerate}
    \item {\em controllable\/} with respect to $L(G)$ and $A_u$;\\ 
      that is, $K A_u\cap L\subseteq K$, and
    \item {\em observable\/} with respect to $L(G)$, $A_o$, and $A_c$;\\ 
      that is, for all $s\in K$ and $\sigma \in A_c$, if $s\sigma \notin K$ and $s\sigma \in L(G)$, then $P^{-1}[P(s)]\sigma \cap K = \emptyset$, where $P:A^*\to A_o^*$.
  \end{enumerate}

  The synchronous product (parallel composition) of languages $L_1\subseteq A_1^*$ and $L_2\subseteq A_2^*$ is defined by 
  \[
    L_1\parallel L_2=P_1^{-1}(L_1) \cap P_2^{-1}(L_2) \subseteq A^*\,, 
  \]
  where $P_i: A^*\to A_i^*$, for $i=1,2$, are projections to local event sets. In terms of generators, it is known that $L(G_1 \| G_2) = L(G_1) \parallel L(G_2)$ and $L_m(G_1 \| G_2)= L_m(G_1) \parallel L_m(G_2)$, see~\cite{CL08} for more details.
  
  We need the following lemma, which should be obvious.
  \begin{lem} 
  \label{obvious}
    For any language $L\subseteq A^*$ and projections $P_1: A^* \to B_1^*$ and $P_2: A^* \to B_2^*$ with $B_2\subseteq B_1\subseteq A$, it holds that $P_1(L) \parallel P_2(L)= P_1(L)$. \hfill\QED
  \end{lem} 

  Let $G$ be a generator over $A$, and let $Q:A^* \to A_o^*$ be a natural projection. A language $K\subseteq L(G)$ is {\em normal\/} with respect to $L(G)$ and $Q$ if $\overline{K} = Q^{-1}Q(K) \cap L(G)$.

  Recall that controllability is preserved by the synchronous product. It is easy to show that the same holds for normality.
  \begin{lem}\label{feng}
    For $i=1,2,\dots ,n$, let $K_i\subseteq L_i$ be controllable with respect to $L_i\subseteq A_i^*$ and $A_{i,u}$, nonconflicting, and normal with respect to $L_i$ and $Q_{i}$, where
    $Q_i: A_i^*\to A_{i,o}^*$ are natural projections that define partial observations in subsystems.
    Then $\parallel_{i=1}^n K_i$ is controllable with respect to $\parallel_{i=1}^n L_i$ and $\cup_{i=1}^n A_{i,u}$ and normal with respect to $\parallel_{i=1}^n L_i$ and $Q$, where
    $Q: (\cup_{i=1}^n A_i)^* \to (\cup_{i=1}^n A_{i,o})^*$ is the natural projection that describes partial observations over the global alphabet. \hfill\QED
  \end{lem}

 Transitivity of controllability and normality is needed later.
  \begin{lem}[\cite{CDC14}]\label{lem_transC}
    Let $K\subseteq L \subseteq M$ be languages over $A$ such that $K$ is controllable with respect to $L$ and $A_u$ and normal with respect to $L$ and $Q$, and $L$ is controllable with respect to $M$ and $A_u$ and normal with respect to $M$ and $Q$. Then $K$ is controllable with respect to $M$ and $A_u$ and normal with respect to $M$ and $Q$.
    \hfill\QED
  \end{lem}
  
  Now we recall the basic notions of coordination control~\cite{JDEDS}. 
  A language $K$ over $\cup_{i=1}^{n}A_i$ is {\em conditionally decomposable with respect to alphabets $(A_i)_{i=1}^{n}$ and $A_k$}, where $\cup_{1\le i,j\le n}^{i\neq j} (A_i\cap A_j) \subseteq A_k\subseteq \cup_{i=1}^{n} A_j$, if 
  \[
    K =\ \parallel_{i=1}^{n} P_{i+k} (K)\,,
  \]
  for projections $P_{i+k}$ from $\cup_{j=1}^{n} A_j$ to $A_i\cup A_k$, for $i=1,2,\ldots,n$.
  The alphabet $A_k$ is a coordinator alphabet and  includes all shared events:
  \[
    A_{sh}=\cup_{1\le i,j\le n}^{i\neq j} (A_i\cap A_j) \subseteq A_k\,. 
  \]
  
  This has the following well-known impact.
  \begin{lem}[\cite{FLT}]\label{lemma:Wonham}
    Let $P_k : A^*\to A_k^*$ be a projection, and let $L_i$ be a language over $A_i$, for $i=1,2,\dots ,n$, and let $A_{sh} \subseteq A_k$. Then $P_k(\parallel_{i=1}^n L_i)=\parallel_{i=1}^n P_k(L_i)$.
    \hfill\QED
  \end{lem}
  
  The problem of coordination control synthesis is now recalled.
  \begin{problem}\label{problem:controlsynthesis}
    Let  $G_i$, for $i=1,2,\dots ,n$, be local generators over the event sets $A_i$ of a modular plant $G=\parallel_{i=1}^n G_i$, and let $G_k$ be a coordinator over an alphabet $A_k$. Let $K\subseteq L(G \| G_k)$ be a specification language. Assume that $A_k\supseteq A_{sh}$ and that $K$ is conditionally decomposable with respect to event sets $(A_i)_{i=1}^{n}$ and $A_k$.

    The overall task $K$ is divided into the local subtasks and the coordinator subtask, cf.~\cite{KvS08}. The supervisor $S_k$ for the coordinator will guarantee that $L(S_k/G_k)\subseteq P_k(K)$. Similarly, the supervisors $S_i$ will guarantee that $L(S_i/ [G_i \| (S_k/G_k) ])\subseteq P_{i+k}(K)$, for $i=1,2,\dots ,n$.

    The problem is to determine the supervisors $S_1, S_2, \dots, S_n$, and $S_k$ such that
    \begin{flalign*}
      && \parallel_{i=1}^n L_m(S_i/ [G_i \| (S_k/G_k) ]) = K\,. && \hfill\triangleleft
    \end{flalign*}
  \end{problem}
  
  \medskip
  The main existential result for a prefix-closed specification $K$ is the special case
of Theorem 13 of \cite{KomendaMS14a} extended to  general $n\geq 2$.
  \begin{theorem}\cite{KomendaMS14a}\label{thm1}
    Consider the setting of Problem~\ref{problem:controlsynthesis}. There exist 
    supervisors $S_1, S_2,\ldots, S_n$ and $S_k$ based on partial observations such that
    \begin{equation}\tag{1}\label{eq:controlsynthesissafety}
        \parallel_{i=1}^n L(S_i/ [G_i \| (S_k/G_k) ]) = K
    \end{equation}
    if and only if $K$ is
    \begin{enumerate}
      \item conditionally controllable with respect to the generators $G_i$ and $G_k$ and the uncontrollable sets $A_{i,u}$ and $A_{k,u}$, for $i=1,2,\ldots,n$, and
      \item conditionally observable with respect to the generators $G_i$ andd $G_k$, the event sets $A_{i,c}$ and $A_{k,c}$, and the projections $Q_{i+k}$ and $Q_{k}$ from $A_i^*$ to $A_{i,o}^*$, for $i=1,2,\ldots,n$. \hfill\QED
    \end{enumerate}
  \end{theorem}

  Recall that $K\subseteq L(G_1\| G_2\| \dots \|G_n \| G_k)$ is {\em conditionally controllable\/} for generators $G_1, G_2, \dots, G_n$ and a coordinator $G_k$ and uncontrollable alphabets $A_{i,u}$, $i=1,2,\dots ,n$, and $A_{k,u}$ if $P_k(K)$ is controllable with respect to $L(G_k)$ and $A_{k,u}$, and $P_{i+k}(K)$ is controllable with respect to $L(G_i) \parallel P_k(K)$ and $A_{i+k,u}=(A_i\cup A_k)\cap A_u$, for $i=1,2,\dots,n$.

  For coordination control with partial observations, the notion of conditional observability is of the same importance as observability for monolithic supervisory control theory with partial observations. We recall that the supervisors $S_i$, $i=1,2,\dots ,n$, are supervisors based on partial observations, because they have only information about observable events from $A_{i,o}$ and observable coordinator events $A_{k,o}$, but do not observe events from $A_{i+k}\setminus (A_{i,o}\cup A_{k,o})$.  
  
  A language $K\subseteq L(G_1\| G_2\| \dots \| G_n\| G_k)$ is {\em conditionally observable\/} with respect to the generators $G_i$ and $G_k$, controllable sets $A_{i,c}$ and $A_{k,c}$, and projections $Q_{i+k}$ and $Q_{k}$, where $Q_i: A_i^*\to A_{i,o}^*$, for $i=1,2,\dots ,n$, if $P_k(K)$ is observable with respect to $L(G_k)$, $A_{k,c}$, $Q_{k}$, and  $P_{i+k}(K)$ is observable with respect to $L(G_i) \parallel P_k(K)$, $A_{i+k,c}=A_c \cap (A_i \cup A_k)$, and $Q_{i+k}$, for $i=1,2,\dots ,n$.
 
  The coordination control theory has been extended to 
  the non-prefix-closed case in~\cite{JDEDS}. The extension consists in 
  introducing coordinators
  for nonblockingness based on abstractions
  that are natural observers.
  We now state an important result from  \cite[Theorem 7]{JDEDS} extended to  general $n\geq 2$.
  \begin{theorem}\label{thm:nonblocking}
    Consider a modular plant with local marked languages $L_i=L_m(G_i)\subseteq A_i^*$, $i=1,\dots ,n$, and let projections $P_k:A_i^*\to (A_i\cap A_k)^*$, with shared events included in $A_k$, be an $L_i$-observer, for $i=1,\dots ,n$. Define $C_k$ as the nonblocking generator with $L_m(C_k)=\parallel_{i=1}^n P_k(L_i)$ with notation $L_k=L_m(C_k)$, i.e., $L(C_k)=\overline{L_k}=\overline{\parallel_{i=1}^n P_k(L_i)}$. Then the coordinated system $G \parallel C_k$  is nonblocking, i.e., $\overline{\parallel_{i=1}^n L_i\parallel L_m(C_k)} = \|_{i=1}^n \overline{L_i} \parallel \overline{L_m(C_k)}$.
  \end{theorem}

\section{Three-level coordination control}\label{sec:sup2cc}
  Since too many events may need to be included in the coordinator alphabet for systems with a large number of subsystems,  the top-down approach with three-level coordination control has been proposed in~\cite{CDC14}.

  Given a modular system $G=G_1 \| G_2 \| \dots \|G_n$, the three-level hierarchical structure depicted in Fig.~\ref{multicoordination} makes it possible to add coordinator events only locally (to low-level group coordinators).
 
  \begin{figure*}
    \centering
    \includegraphics[scale=1]{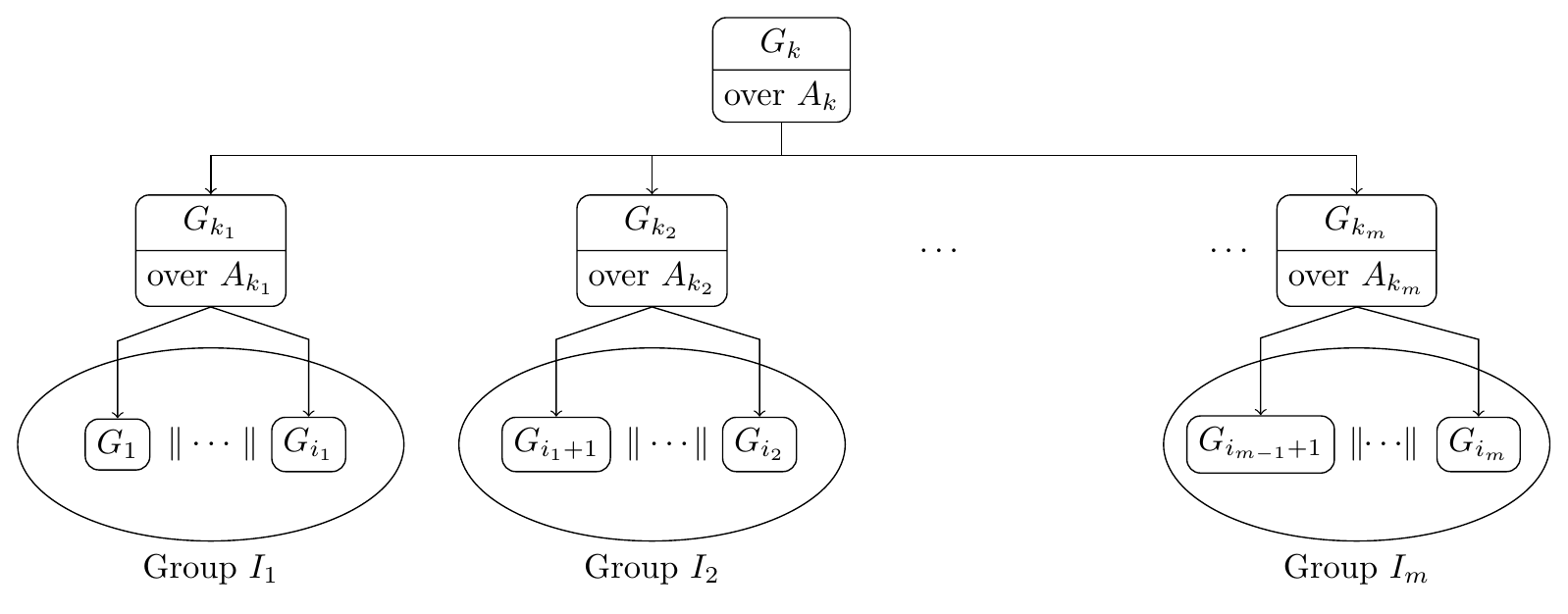}
    \caption{The multilevel control architecture under consideration.}
    \label{multicoordination}
  \end{figure*}

  The event sets of low-level groups $I_j$, $j=1,2, \dots ,m$, are denoted by 
  \[
    A_{I_j}=\bigcup\nolimits_{i\in I_j} A_i\,.
  \] 
  Recall that $P_{I_r}$ denotes the projection $P_{I_r}:A^*\to A_{I_r}^*$. Then $P_{I_r+k}: A^*\to (A_{I_r}\cup A_k)^*$ stands for the projection to the group alphabets extended with the high-level coordinator events. Similarly, $P_{j+k_r+k}: A^*\to (A_j\cup A_{k_r}\cup A_k )^*$ denotes the projection to the alphabet $A_j$ of an automaton $G_j$ belonging to the group $I_r$ extended with the alphabet $A_{k_r}$ of the group coordinator of the low-level group $I_r$ and the high-level coordinator alphabet $A_k$. 

  We start by constructing $A_k\subseteq A_{sh}=\bigcup\nolimits_{k,\ell\in\{1,2, \dots ,m\}}^{k\not =l} (A_{I_k}\cap A_{I_\ell})$  such that $K =\ \parallel_{r=1}^m P_{I_r+k} (K)$. Note that $A_{sh}$, that is, the set of events shared by the low-level groups, is  much smaller than the set of all shared events. The reason is that the events shared only among subsystems belonging to a given low-level group do not count for $A_{sh}$. An algorithm to construct $A_k$ as an extension of $A_{sh}$ making the first equation of Definition~\ref{mcd} below hold true is described in~\cite{SCL12}.

  In order to simplify the notation and definitions, we have included in~\cite{CDC14} into the group coordinator alphabets $A_{k_j}$ all events from the global coordinator by defining $A_{k_j} := A_{k_j}\cup A_k$, for $j=1,2, \dots, m$. This simplification enables us to use only the group coordinators $G_{k_j}$ in all the definitions below, which is more concise than using $G_{k_j} \| G_k$, but we have to bear in mind that from now on $G_{k_j}$ may also contain the high-level coordinator events from other groups than $I_j$.

  \begin{definition}[3-level conditional decomposability]\cite{CDC14}\label{mcd}$ $\\
    A language $K\subseteq A^*$ is said to be {\em three-level conditionally decomposable\/} with respect to the  alphabets $A_1$, $A_2$, \ldots, $A_n$, the high-level coordinator alphabet $A_{k}$, and the low-level coordinator alphabets $A_{k_1}$, $A_{k_2}$, \ldots, $A_{k_m}$ if
    \begin{align*} 
      K =\ \parallel_{j=1}^m P_{I_j+k} (K) && \text{ and } && 
      P_{I_j+k} (K) &=\ \parallel_{i\in I_j} P_{i+k_j} (K)\,
    \end{align*}
    for $j=1,2,\dots,m$.
    $\hfill\triangleleft$
  \end{definition} 
  
  Definition~\ref{mcd} makes sense, because on the right-hand side of the second equation $P_{i+k_j} (K)$ includes all events from both the group coordinator $A_{k_j}$ and the high-level coordinator $A_{k}$. 

  \begin{problem}[Three-level coordination control problem]
  \label{problem:2-controlsynthesis} $ $\\
    Consider the modular system $G=G_1 \| G_2 \| \dots \| G_n$ along with the three-level hierarchical structure of the subsystems organized into groups $I_j$, $j=1,2,\dots ,m\leq n$, on the low level. 
    The synchronous products $\parallel_{i\in I_j} G_i$, $j=1,2,\dots ,m$, then represent the $m$ high-level systems. 
    The coordinators $G_{k_j}$ are associated to groups of subsystems $\{G_i \mid i\in I_j\}$, $j=1,2,\dots, m$.  The three-level coordination control problem consists in synthesizing the supervisor $S_i$ for every low-level system $G_i$, $i=1,2,\ldots,n$, and the high-level supervisor $S_{k_j}$ supervising the group coordinator $G_{k_j}$, $j=1,2,\dots, m$, such that the specification $K=\overline{K}\subseteq L(G)$ is met by the closed-loop system, i.e.,
    \begin{flalign*}
      && \parallel_{j=1}^m  \parallel_{i\in I_j} L(S_i/  [G_i\| (S_{k_j}/G_{k_j})]) = K\,. && \triangleleft
    \end{flalign*}
  \end{problem}
  \medskip

  Low level (group) coordinators $G_{k_j}$, $j=1,2, \dots ,m$, are computed using Algorithm~\ref{alg1} below.
  \begin{algorithm}
    \caption{Computation of the group coordinators.}
    \label{alg1}
    For a specification $K$, the coordinator $G_{k_j}$ of the $j$-th group of subsystems $\{G_i \mid i\in I_j\}$ is computed as follows.
    \begin{enumerate}
      \item Set $A_{k_j} = \bigcup_{k,\ell\in I_j}^{k\neq \ell} (A_k\cap A_\ell)$ to be the set of all shared events of systems from the group $I_j$.
      \item Extend $A_{k_j}$ so that $P_{I_r+k}(K)$ is conditional decomposable with respect to $(A_i)_{i\in I_j}$ and $A_{k_j}$, for instance using a method described in~\cite{SCL12}.
      \item Set the coordinator equal to $G_{k_j}= \|_{i= 1}^{n} P_{k_j}(G_i)$.
    \end{enumerate}
  \end{algorithm}
  Recall that due to the extension of $A_{k_j}$ by high-level coordinator events, $A_k\subseteq A_{k_j}$, hence $L(G_k) \| L(G_{k_j})$ of~\cite{KMvS13} is reduced to $L(G_{k_j})$. Indeed, by our choice of the coordinators, $L(G_k) \| L(G_{k_j}) = P_k(L) \parallel P_{k_j}(L) = P_{k_j}(L) = L(G_{k_j})$, where $L=\parallel_{i= 1}^{n} L(G_i)$ is the global plant language and the second equality holds by Lemma~\ref{obvious}. Therefore, instead of the low-level coordinators $G_{k_j}$, $j=1,2,\dots ,m$, for subsystems belonging to the individual groups $\{G_i \mid i\in I_j\}$ and the high-level coordinators $G_k$ that coordinate the different groups, we are using only the low-level (group) coordinators $G_{k_j}$, but over larger alphabets compared to~\cite{KMvS13}. 

  Since the only known condition ensuring that the projected generator is smaller than the original one is the observer property~\cite{WW96} we might need to further extend the alphabets $A_{k_j}$ so that the projection $P_{k_j}$ is an $L(G_i)$-observer, for all $i\in I_j$. 
  
  The key concept is the following.
  \begin{definition}[\cite{KMvS13}]\label{def:2-conditionalcontrollability}
    Consider the setting and notations of Problem~\ref{problem:2-controlsynthesis}, and let $G_k$ be a coordinator. A language $K\subseteq L(\parallel_{i=1}^{n} G_i)$ is {\em three-level conditionally controllable\/} with respect to the generators $G_1$, $G_2$, \dots, $G_n$, the local alphabets $A_1$, $A_2$, \ldots, $A_n$, the low-level coordinator alphabets $A_{k_1}$, $A_{k_2}$, \dots, $A_{k_m}$, and the uncontrollable alphabet $A_{u}$ if for all $j=1,2,\dots,m$ 
    \begin{enumerate}
      \item\label{cc1} $P_{k_j}(K)$ is controllable with respect to $L(G_{k_j})$ and $A_{k_j,u}$,
      \item\label{cc2} $P_{i+k_j}(K)$ is controllable with respect to  $L(G_i) \parallel P_{k_j}(K)$ and $A_{i+k_j,u}$, for all $i\in I_j$.
      $\hfill\triangleleft$
    \end{enumerate}
  \end{definition}
  
 For the sake of brevity, $K$ will be called three-level conditionally controllable with respect to $G_i$, $i\in I_\ell$, and $G_{k_\ell}$, where some sets are not referenced.

  For multilevel systems with partial observations, three-level conditionally observability, cf.~\cite{ACC15}, is needed. Unfortunately, it is not closed under language unions and, therefore,  three-level conditional normality has been proposed in \cite{ACC15}, where it is shown that the supremal three-level conditionally normal language always exists.
  
  \begin{definition}\label{def:2-conditionalnormality}
    A language $K\subseteq L(\parallel_{i=1}^{n} G_i)$ is {\em three-level conditionally normal\/}
    with respect to the generators $G_1$, $G_2$, \dots, $G_n$, the local alphabets $A_1$, $A_2$, \ldots, $A_n$, the low-level coordinator alphabets $A_{k_1}$, $A_{k_2}$, \dots, $A_{k_m}$, and the corresponding natural projections if for all $j=1,2,\dots,m$
    \begin{enumerate}
      \item\label{cn1} $P_{k_j}(K)$ is normal with respect to $L(G_{k_j})$ and $Q_{k_j}$,
      \item\label{cn2} $P_{i+k_j}(K)$ is normal with respect to $L(G_i) \parallel P_{k_j}(K)$ and $Q_{i+k_j}$, for all $i\in I_j$.
      $\hfill\triangleleft$
    \end{enumerate}
  \end{definition}
  
  The computation of the supremal three-level conditionally controllable and conditionally normal sublanguage of $K$, denoted by $\suptwocCN(K, L, A, Q)$, has been studied in~\cite{ACC15}. We have shown that under some controllability and normality conditions on all coordinator alphabets it can be computed in a distributed way based on the following languages corresponding to supervisors for low-level group coordinators and local supervisors for individual subsystems, respectively. For all $j=1,2,\dots,m$ and $i\in I_j$,
  \begin{align}\label{sup2cc}
      \supCN_{k_j} & = \supCN(P_{k_j}(K),L(G_{k_j}), A_{k_j,u}, Q_{k_j}) \\
      \supCN_{i+k_j} & = \supCN(P_{i+k_j}(K), L(G_i) \| \supCN_{k_j}, A_{i+k_j,u}, Q_{i+k_j}) \nonumber
  \end{align}
  where $\supCN(K,L,A_u, Q)$ denotes the supremal sublanguage of $K$ controllable  with respect to $L$ and $A_u$ and normal with respect to $L$ and the natural projection $Q$, see~\cite{CL08}.

  As in the centralized coordination, the following inclusion always holds true.

  \begin{lem}\label{inclusion}
    For all $j=1,2,\dots ,m$ and for all $i\in I_j$, we have that 
    $P^{i+k_j}_{k_j}(\supCN_{i+k_j})\subseteq \supCN_{k_j}$. 
  \end{lem}
  \begin{proof}
    The proof follows immediately from the definition of $\supCN_{i+k_j}$. Indeed, we have that $P^{i+k_j}_{k_j} (\supCN_{i+k_j}) \subseteq \supCN_{k_j}$, because $\supCN_{k_j}$ is part of the plant language of $ \supCN_{i+k_j}$ over the alphabet $A_{k_j}$.
  \end{proof}
  We recall the notation for the closed-loop corresponding
to group $I_j$, i.e. $\supcCN_j =\ \parallel_{i\in I_j} \supCN_{i+k_j}$ for $j=1,2,\dots ,m$.
  The  main result of \cite{ACC15} is now recalled. 
  \begin{theorem}[\cite{ACC15}]
    \label{thm:construction}
    Consider Problem~\ref{problem:2-controlsynthesis} and the languages defined in~(\ref{sup2cc}). 
    For $j=1,2,\dots ,m$ and $i\in I_j$, 
    let the languages $P^{i+k_j}_{k_j} (\supCN_{i+k_j})$ be controllable with respect to $L(G_{k_j})$ and $A_{k_j,u}$, and 
    normal with respect to $L(G_{k_j})$ and $Q_{k_j}$, and 
    let $P_k^{I_j}(\supcCN_j)$ be controllable with respect to $L(G_k)$ and $A_{k,u}$, and normal with respect to $L(G_k)$ and $Q_{k}$. Then 
    \[
      \suptwocCN(K, L, A, Q) =\ \parallel_{j=1}^m \parallel_{i\in I_j} \supCN_{i+k_j}\,.
    \]
    \hfill\QED
  \end{theorem}

\section{Combined Approach to Multilevel Coordination Control of Modular DES}\label{sec:combined}
  Recently, we have proposed two different constructive approaches to multilevel supervisory control: bottom-up~\cite{ECC14} and top-down~\cite{CDC14}. Bottom-up approach relies only on original notions of conditional decomposability and conditional controllability of the specification language,  while top-down approach requires the specification to be conditionally decomposable and conditionally controllable with respect to the multilevel architecture. 
  In the top-down approach, the specification is decomposed a priori in the top-down manner: firstly, with respect to the high-level coordinator alphabet and then with respect to the group coordinators for all low-level groups of subsystems. The advantage of the top-down approach is that, for prefix-closed specifications, the computation at the lowest level consists in constructing supervisors for individual subsystems and no further computation at the higher level is needed.

  However, the least restrictive supervisors 
can only be computed under some conditions. We have presented in~\cite{CDC14} the sufficient conditions for distributed computation of full observation supervisors yielding the maximally permissive solution in the three-level hierarchical control architecture. This condition has been generalized in~\cite{ACC15} in two directions: to partial observations and to weaker sufficient conditions for the distributed computation of local supervisors assisted by coordinators. These weaker sufficient conditions are homogeneous, i.e., they are both formulated in terms of controllability and normality for both hierarchical interfaces: between the low level and the middle level and between the middle level and the top level.

In this section all languages are assumed to be prefix-closed.
  In the general case with non-prefix-closed specifications, the individual supervisors of the groups can be conflicting and also the group supervisors on the higher level might be conflicting. Therefore, additional coordinators for nonblocking should be constructed at all levels, which is presented in the next section. 

  To conclude, the main drawback of the top-down approach is the lack of generality: the blocking issue and the restrictive conditions for a distributed computation of the maximally permissive solution: supremal three-level conditionally controllable sublanguages.

  In this paper, we propose a combined approach that can be described as a top-down decomposition followed by a bottom-up computation. This proposed approach combines the strong features of both approaches, namely the low complexity of the top-down approach with the generality of the bottom-up approach that enables effective synthesis of both a posteriori supervisors to make sufficient conditions for distributed computation of supervisors hold and of coordinators for nonblocking.

  It is then natural to impose controllability and normality of low-level supervisors with respect to group coordinators and also controllability and normality of group supervisors with respect to the high coordinator at the very top level. In this paper, we will show that these supervisors can be synthesized in the bottom-up manner, i.e., we start with the supervisors on coordinator alphabets of each low-level group.

  In the case that controllability of the projected low-level supervisors with respect to the group coordinators and/or controllability of projected group supervisors with respect to the top coordinator from Theorem~\ref{thm:construction} do not hold, a posteriori supervisors on respective coordinator alphabets can be synthesized to make these conditions hold. 

  We will show that both a posteriori supervisors and coordinators for nonblocking can be computed in the bottom-up manner. This is the main message of this paper: first, we perform a top-down design of coordinators based on two-level decomposition of the specification and this top-down design is followed by a bottom-up computation of a posteriori supervisors and coordinators for nonblocking.

It is easy to shown that the language  $\parallel_{j=1}^m \parallel_{i\in I_j} \supCN_{i+k_j}$ of Theorem~\ref{thm:construction} further restricted by a posteriori supervisors will always satisfy all controllability and normality conditions required in Theorem~\ref{thm:construction}.
 It appears that controllability and normality conditions on the low-level coordinator alphabets and on the high-level coordinator alphabet can be imposed by a posteriori supervisors defined a follows.

  We first compute a posteriori supervisors on the low-level  coordinator alphabets $A_{k_j}$, $j=1,2,\dots,m$, by
    \begin{equation}
    \label{aposteriori_low}
    \widetilde{  \supCN_{k_j}} = \cap_{i\in I_j} \supCN(P_{k_j}(\supCN_{i+k_j}), L(G_{k_j}), A_{k_j,u},Q_{k_j}).
\end{equation}
  This supervisor will guarantee controllability and normality with respect to the group coordinator alphabets as required in Theorem~\ref{thm:construction}. It should be noticed that
  \begin{equation}
  \label{distr_computation}
    \widetilde{ \supCN_{k_j}} = \supCN(P_{k_j}(\parallel_{i\in I_j}\supCN_{i+k_j}), L(G_{k_j}), A_{k_j,u},Q_{k_j}),
  \end{equation}
  but the former distributed form is more suitable for computation of a posteriori supervisors $\widetilde{  \supCN_{k_j}}$ on group coordinator alphabets because of obvious complexity reasons. Otherwise stated, the a posteriori supervisors can be distributed and their roles consist simply in replacing local supervisors for individual subsystems $G_i$ at the lowest level: $\supCN_{i+k_j}$ by 
  \begin{align}
  \label{supervisor_resulting_low}
    \supCN_{i+k_j} \parallel \widetilde{  \supCN_{k_j}}&=\\
  \supCN_{i+k_j} \parallel & \cap_{i\in I_j} \supCN(P_{k_j}(\supCN_{i+k_j}), L(G_{k_j}))
\nonumber.
  \end{align}

  Moreover, we show in Theorem~\ref{aposteriori}  that the restriction induced by the supervisor does alter maximal permissiveness. Then we compute the a posteriori supervisor on the high-level coordinator alphabet by
  \[
   \widetilde{  \supCN_k }= \supCN(P_k( \parallel_{j=1}^m \supcCN_j), L(G_k), A_{k,u},Q_{k}),
  \]
  where $\supcCN_j=\parallel_{i\in I_j} \supCN_{i+k_j} \parallel  \widetilde{  \supCN_{k_j}}$ is the resulting group supervisor. The supervisor $\widetilde{  \supCN_k}$ will guarantee controllability and normality with respect to the high-level coordinator $L(G_k)$.

  Note that it is easy to see that $\widetilde{ \supCN_k}$ can be computed in the modular way as follows:
  \begin{align}
  \label{aposteriori_high}
    \widetilde{ \supCN_k} & = 
 \parallel_{j=1}^m  \supCN(P_k( \supcCN_j), L(G_k), A_{k,u},Q_{k})
  \end{align}
  This is a very special case of modular control with multiple prefix-closed specifications \cite{IEEETAC08} for a single plant $G_k$. Therefore it follows from the assumption that all languages involved are prefix-closed, hence the languages in the intersection are trivially nonconflicting, which is required for preserving normality and controllability under intersection. 

  It can be shown that the language $M$ further restricted by these supervisors will always satisfy all controllability and normality conditions required in Theorem~\ref{thm:construction}. Somewhat surprisingly, it can be shown that these a posteriori supervisors do not alter another important property: supremality. The result below shows that the solution is still minimally restrictive with respect to our two level coordination control architecture, which is formally shown in the second inclusion
  of the proof.

  \begin{theorem}\label{aposteriori}
    Consider the setting of Theorem~\ref{thm:construction}. Then
    \begin{multline*}
      \suptwocCN(K, L, A, Q)\\
        = (\parallel_{j=1}^m (\parallel_{i\in I_j} \supCN_{i+k_j})
          \parallel \widetilde{  \supCN_{k_j}})\parallel \widetilde{  \supCN_k}\,
    \end{multline*}
    where a posteriori supervisors $\widetilde{  \supCN_{k_j}}$ and
 $\widetilde{  \supCN_k}$ are defined in equations (\ref{aposteriori_low}) and (\ref{aposteriori_high}), respectively.
  \end{theorem}
  \begin{proof}
    For simplicity, denote $\suptwocCN(K, L, A, Q) = \suptwocCN$, and let us use the notation 
    \[
      M_j = \supcCN_j =\ \parallel_{i\in I_j} \supCN_{i+k_j} \parallel  \widetilde{  \supCN_{k_j}}
    \]
    for the resulting language of the (centralized) coordination control for each group $I_j$, $j=1,2,\dots ,m$. 
    We denote
    \[
      M=\ \parallel_{j=1}^m M_j \parallel \widetilde{  \supCN_k }\,. 
    \]
    Hence, we need to show that $\suptwocCN=\parallel_{j=1}^m M_j$.
    
    In order to show the inclusion $M \subseteq \suptwocCN$, it suffices to prove that $M$ is three-level conditionally controllable and conditionally normal with respect to $G_i$, $i\in I_\ell$, and $G_{k_\ell}$, for $\ell = 1,2,\dots, m$. Then, since both $M$ and $\suptwocCN$ are sublanguages of $K$, and $\suptwocCN$ is the supremal one having these properties, it will follow that $M \subseteq \suptwocCN$.
    
    For items 1 of three-level conditional controllability and conditional normality, we show that, for any $j=1,2,\dots, m$, $M_j$ is conditionally controllable and conditionally normal with respect to $G_i$, $i\in I_j$, $L(G_{k_j})$, $A_{k_j,u}$, and $Q_{k_j}$. First, note that
    \begin{eqnarray*}\label{distribution}
      P_{k_j}(M) &= P_{k_j}(\parallel_{\ell = 1}^m M_\ell  \parallel  
\widetilde{ \supCN_k}) =\\ & P_{k_j}(M_j) \parallel \parallel_{\ell=1,2,\dots ,m}^{\ell \neq j} P_{k_j}(M_\ell)
\parallel   \widetilde{ \supCN_k}\,
    \end{eqnarray*}
    because $A_{k_j}\supseteq A_k$ and  $A_{k_j}$ contains all shared events in the
composition.
    Moreover, $P_{k_j}(M_j) = P_{k_j}(\|_{i\in I_j} \supCN_{i+k_j} \parallel  
\widetilde{ \supCN_{k_j}} ) = \cap_{i\in I_j} P_{k_j} (\supCN_{i+k_j}) \cap  \tilde \supCN_{k_j}$ , because of Lemma~\ref{lemma:Wonham} and the fact that $A_{k_j}$ contains all shared events of subsystems of the group $I_{j}$).

 It is then easy to see that $P_{k_j}(M_j)=\widetilde{\supCN_{k_j}}$ is controllable and normal with respect to $L(G_{k_j})$, $A_{k_j,u}$ and $Q_{k_j}$.
We now show $M_j$, $j=1,2,\dots, m$, are conditionally controllable and conditionally normal with respect to their groups $G_i$, $i\in I_j$, and $G_{k_j}$.

   Since the distributivity holds due to Lemma~\ref{lemma:Wonham}, $P_{i+k_j}(M_j) = P_{i+k_j}(\|_{i'\in I_j} \supCN_{i'+k_j} \parallel  \widetilde{\supCN_{k_j} } ) = \|_{i'\in I_j} P_{k_j} (\supCN_{i'+k_j}) \parallel  \widetilde{ \supCN_{k_j}} = \supCN_{i+k_j} \parallel \|_{i'\in I_j}^{i\neq i'} P_{k_j} (\supCN_{i'+k_j}) \parallel  \widetilde{ \supCN_{k_j}}$. Observe that 
    \[
      P_{i+k_j}(M_j) = \supCN_{i+k_j} \parallel P_{k_j}(M_j)\,,
    \]
    since $\supCN_{i+k_j} \parallel P_{k_j} (\|_{i'\in I_j} \supCN_{i'+k_j}
\parallel \widetilde{ \supCN_{k_j}}) = \supCN_{i+k_j} \parallel \|_{i'\in I_j} P_{k_j} (\supCN_{i'+k_j}) \parallel \widetilde{ \supCN_{k_j}} = \supCN_{i+k_j} \parallel \|_{i'\in I_j}^{i\neq i'} P_{k_j} (\supCN_{i'+k_j}) \parallel \widetilde{ \supCN_{k_j}}=
P_{i+k_j}(M_j)$. 

Therefore, by Lemma~\ref{feng}, $P_{i+k_j}(M_j)$ is controllable and normal with respect to $[L(G_i) \parallel \supCN_{k_j}] \parallel P_{k_j}(M_j) = L(G_i) \parallel P_{k_j}(M_j)$, where the last equality is by the fact that $P_{k_j} (M_j) \subseteq \supCN_{k_j}$, for any $j=1,2,\dots ,m$ and $i\in I_j$.

Altogether, $M_j$, $j=1,2,\dots, m$, are conditionally controllable and conditionally normal with respect to their groups $G_i$, $i\in I_j$, and $G_{k_j}$.

    Furthermore, for $\ell=1,2, \dots ,m$, $\ell\neq j$, 
    \begin{align}\label{eq1}
      P_{k_j}(M_\ell) = P_k(M_\ell)\,, 
    \end{align}
    because 
    $M_\ell \subseteq A_{I_\ell}^*$, 
    $A_k \subseteq A_{k_j} \subseteq A_{I_j}\cup A_k$, 
    $A_{I_j} \cap A_{I_\ell}\subseteq A_k$, whence
    $A_{k_j}\cap A_{I_\ell} = A_k \cap A_{I_\ell}$. 
 
Now, we have
 $P_{k}(M_\ell) \parallel  \widetilde{  \supCN_k}= P_k ( \parallel_{i\in I_{\ell}} \supCN_{i+k_{\ell}} \parallel  \widetilde{  \supCN_{k_{\ell}}} \parallel  \widetilde{  \supCN_k}=\parallel  \widetilde{  \supCN_k}$.
    
 This is because
 \begin{align}
\widetilde{  \supCN_k}&=\parallel_{j=1}^m  \supCN(P_k( \supcCN_j), L(G_k), A_{k,u},Q_{k})\\
 &=\parallel_{j=1}^m  \supCN(P_k(\parallel_{i\in I_j} \widetilde{  \supCN_{i+k_j}}), L(G_k), A_{k,u},Q_{k})\\
& = \parallel_{j=1}^m  \cap_{i\in I_j} \supCN(P_k(\widetilde{  \supCN_{i+k_j}}), L(G_k), A_{k,u},Q_{k})\,.\nonumber
 \end{align}
  
Therefore,
$P_k(M_\ell) \parallel  \widetilde{  \supCN_k}=  \widetilde{  \supCN_k}$ are controllable and normal with respect to $L(G_k)$, $A_{k,u}$, and $Q_{k}$, for $\ell = 1,2,\dots ,m$.

    Altogether, in accordance with Lemma~\ref{feng}, we obtain that $P_{k_j}(M) =\ \parallel_{\ell=1}^m P_{k_j}(M_\ell) \parallel P_{k_j}(\widetilde{  \supCN_k })= P_{k_j}(M_j) \parallel \parallel_{\ell = 1,2,\dots ,m}^{\ell \neq j} P_{k_j}(M_\ell)
\parallel \widetilde{  \supCN_k }$ is controllable and normal with respect to $L(G_{k_j}) \parallel \parallel L(G_k ).$ We recall that  $L(G_{k_j}) \parallel L(G_{k}) = L(G_{k_j})$. Therefore, $P_{k_j}(M)$ is controllable with respect to $L(G_{k_j})$ and $A_{k_j,u}$, and normal with respect to $L(G_{k_j})$ and $Q_{k_j}$. This shows items~\ref{cc1} of both three-level conditional controllability and conditional normality.


  In order to show item 2 of three-level conditional controllability and 
  conditional normality,  it must be shown that
  $P_{i+k_j}(M) = P_{i+k_j}(\parallel_{\ell=1}^m M_\ell \parallel 
 \widetilde{  \supCN_k }  )$ 
  is controllable with respect to $L(G_i) \parallel P_{k_j}(M)$ and $A_{i+k_j,u}$, and normal with respect to $L(G_i) \parallel P_{k_j}(M)$ and $Q_{i+k_j}$.
  Note that $P_{k_j}(M)=P_{k_j}(M_j) \parallel \parallel_{\ell=1,2,\dots ,m}^{\ell \neq j} P_{k_j}(M_{\ell})\parallel  \widetilde{  \supCN_k }$, because due to 
  $A_{k_j}\subseteq A_k$ we have
$ P_{k_j}( \widetilde{  \supCN_k }= \widetilde{  \supCN_k }$.

  In a similar way as above, we get
  \begin{align*}
    P_{i+k_j}(M) & = P_{i+k_j}(M_j) \parallel \parallel_{\ell=1,2,\dots ,m}^{\ell \neq j} P_{i+k_j}(M_\ell)  \parallel  P_{i+k_j}(\widetilde{  \supCN_k })  \\
                 & = P_{i+k_j}(M_j) \parallel \parallel_{\ell=1,2,\dots ,m}^{\ell \neq j} P_{k_j}( M_\ell ) \parallel \widetilde{  \supCN_k }
  \end{align*}
  since, for $j\neq\ell$, $A_{I_j}\cap A_{I_\ell}\subseteq A_k \subseteq A_{k_j}$ fulfills the requirements of Lemma~\ref{lemma:Wonham}, which justifies the first equation. Moreover, it also implies that $P_{i+k_j}(M_\ell)=P_k(M_\ell)=P_{k_j}(M_\ell)$, see equation~(\ref{eq1}), which justifies the second equation.
  Furthermore, from above 
We recall at this point that $M_j$ are conditionally controllable and conditionally normal with respect to their groups $G_i$, $i\in I_j$, the group coordinators $L(G_{k_j})$, whence for all $j=1,2,\dots ,m$ and for all $i \in I_j$ we have that 
  we have that $P_{i+k_j}(M_j)$ are controllable and normal with respect to $L(G_i)\parallel P_{k_j}(M_j)$, $A_{i+k_j,u}$, and $Q_{i+k_j}$. 
  It is obvious that languages $P_{k_j}(M_\ell)$ for $\ell=1,2,\dots ,m$, $\ell \neq j$, are controllable and normal with respect to themselves.
  Finally, $\widetilde{  \supCN_k }$ is controllable normal with respect to 
  itself. 

  Therefore, according to Lemma \ref{feng}, $P_{i+k_j}(M)$ is controllable and normal with respect to $[L(G_i)\parallel P_{k_j}(M_j)] \parallel \|_{\ell=1,2,\dots ,m}^{\ell \neq j} P_{k_j}(M_\ell) \parallel  \widetilde{  \supCN_k }=L(G_i) \parallel \|_{\ell=1}^{m} P_{k_j}(M_\ell)  \parallel  \widetilde{  \supCN_k } = L(G_i) \parallel P_{k_j} (\|_{\ell=1}^{m} M_\ell)  \parallel  \widetilde{  \supCN_k }=L(G_i) \parallel P_{k_j}(M)$, $A_{i+k_j,u}$, and $Q_{i+k_j}$, which was to be shown. 
  Note that distributivity $P_{i+k_j} (\parallel_{\ell=1}^{m} M_\ell) =\ \parallel_{\ell=1}^{m} P_{i+k_j}(M_\ell)$ holds true in accordance with Lemma~\ref{lemma:Wonham}, because $A_{i+k_j}$ contains $A_{i+k}$ and $A_{i+k}$ contains all shared events of languages $P_{i+k_j}(M_\ell)$ over their respective alphabets $A_{I_\ell+i}$, $\ell=1,2,\dots ,m$. 
  More precisely, for $i\in I_j$ we have that
  \begin{eqnarray*} 
    A_{I_\ell+i}=
      \left\{
        \begin{array}{ll}
          A_{I_j} &  
             \hbox{if } \ell=j\\        
        A_{I_\ell}  &  \hbox{otherwise}
       \end{array}
       \right. 
  \end{eqnarray*}


  The converse $\suptwocCN \subseteq (\parallel_{j=1}^m~(\parallel_{i\in I_j}~\supCN_{i+k_j}) \parallel \widetilde{supCN_{k_j}})\parallel supCN'_k$ will be proven by showing that for all $j=1,2,\dots ,m$ and for all $i\in I_j$, 
  \begin{equation}
  \label{incl}
    P_{i+k_j}(\suptwocCN) \subseteq \supCN_{i+k_j} \parallel \widetilde{  \supCN_{k_j}} \parallel supCN'_k\,.
  \end{equation} 

  According to the definition of synchronous product,
  Eq.~\ref{incl} is equivalent to three separate inclusions
  \begin{itemize}
  \item[(i)]
    $P_{i+k_j}(\suptwocCN) \subseteq \supCN_{i+k_j}$
  \item[(ii)]
    $P_{i+k_j}(\suptwocCN) \subseteq (P_{k_j}^{i+k_j})^{-1} \widetilde{  \supCN_{k_j}} $
  \item[(iii)]
    $P_{i+k_j}(\suptwocCN) \subseteq (P_{k}^{i+k_j})^{-1} \supCN'_k $
  \end{itemize}

  The first inclusion is not hard to see. Indeed, from the definitions of conditional controllability and conditional normality, $P_{i+k_j}(\suptwocCN)$ is controllable and normal with respect to $L(G_i)\parallel P_{k_j}(\suptwocCN)$,  $A_{i+k_j,u}$, and $Q_{i+k_j}$. Furthermore, $L(G_i)\parallel P_{k_j}(\suptwocCN)$ is controllable and normal with respect to $L(G_i) \parallel \supCN_{k_j}$, $A_{i+k_j,u}$, and $Q_{i+k_j}$, because $P_{k_j}(\suptwocCN)$ being controllable and normal with respect to $L(G_{k_j})$ is also controllable and normal with respect to the smaller language $\supCN_{k_j} \subseteq L(G_{k_j})$. Therefore, using transitivity of controllability and normality (Lemma~\ref{lem_transC}), $P_{i+k_j}(\suptwocCN)$ is controllable and normal with respect to $L(G_i) \parallel \supCN_{k_j}$, $A_{i+k_j,u}$ and $Q_{i+k_j}$.

  The proof of the other two inclusions is more involved. First, note that (ii) is equivalent to the inclusion $P^{i+k_j}_{k_j}  P_{i+k_j}(\suptwocCN) \subseteq \widetilde{  \supCN_{k_j}}$,
  and that $P_{k_j}(\suptwocCN)=P^{i+k_j}_{k_j}  P_{i+k_j}(\suptwocCN)$.
  Hence, it is equivalent to the inclusion $P_{k_j}(\suptwocCN) \subseteq \widetilde{  \supCN_{k_j}}$. We recall at this point that $\widetilde{  \supCN_{k_j}} =\ \parallel_{i\in I_j} \supCN(P_{k_j}(\supCN_{i+k_j}), L(G_{k_j}), A_{k_j,u},Q_{k_j})$. By the definition of the three-level conditionally controllable and normal languages, $P_{k_j}(\suptwocCN)$ is controllable with respect to $L(G_{k_j})$ and normal with respect to $L(G_{k_j})$ and $Q_{k_j}$. Clearly, $P_{k_j}(\suptwocCN)\subseteq P_{k_j}(K)$. Now, $\supCN(P_{k_j}(\supCN_{i+k_j}), L(G_{k_j}), A_{k_j,u},Q_{k_j})$ is the supremal sublanguage of $P_{k_j}(\supCN_{i+k_j})$, which is controllable and normal with respect to $L(G_{k_j})$ and $Q_{k_j}$. Hence, we obtain that $P_{k_j}(\suptwocCN) \subseteq \widetilde{ \supCN_{k_j}}$ provided $P_{k_j}(\suptwocCN)$ is also a sublanguage of $P_{k_j}(\supCN_{i+k_j})$. Thus, it remains to show that $P_{k_j}(\suptwocCN)\subseteq P_{k_j}(\supCN_{i+k_j})$. 
  However, it holds that $P_{i+k_j}(\suptwocCN)\subseteq \supCN_{i+k_j}$, because $P_{i+k_j}(\suptwocCN)$ is, by definition of the three-level conditionally controllable and normal languages, a sublanguage of $P_{i+k_j}(K)$ that is controllable and normal with respect to $L(G_i) \parallel P_{k_j}(\suptwocCN)$ and $Q_{k_j}$, i.e., it is by transitivity of Lemma~\ref{lem_transC} (and the fact that the synchronous product preserve both controllability and normality for nonconflicting languages) controllable and normal with respect to $L(G_i) \parallel L(G_{k_j})$ and $Q_{k_j}$. Since $\supCN_{k_j}\subseteq L(G_{k_j})$, we obtain that $P_{i+k_j}(\suptwocCN)$ is controllable and normal with respect to $L(G_i) \parallel \supCN_{k_j}$ and $Q_{k_j}$. Therefore, $P_{i+k_j}(\suptwocCN)$ has to be included in $\supCN_{i+k_j}$, which is the supremal sublanguage of $P_{i+k_j}(K)$ that is controllable and normal with respect to  $L(G_i) \parallel \supCN_{k_j}$ and $Q_{k_j}$.

  Finally, inclusion (iii) can be shown using the same arguments as in (ii).
\end{proof}

\section{General Case: A Posteriori Supervisors Combined with Coordinators for Nonblocking}\label{sec:general}

  In the previous section we have shown that a posteriori supervisors enable us to compute maximally permissive supervisors for our three-level coordination control architecture whenever there is no problem with blocking, e.g., in the prefix-closed case. It is clear from Theorem~\ref{aposteriori} that first the a posteriori supervisors on the group coordinator alphabets 
$\widetilde{  \supCN_{k_j}}$ are computed and then the a posteriori supervisor 
$\widetilde{  \supCN_k}$ on the high-level coordinator alphabet is computed. Otherwise stated, the computation of the a posteriori supervisors goes in the bottom-up way. The computation of these supervisors is necessary for obtaining the maximally permissive solution, i.e., the supremal three-level conditionally controllable and conditionally-normal sublanguage of the specification if the sufficient condition of Theorem~\ref{thm:construction} does not hold.

  In the general case, local supervisors $\supCN_{i+k_j}$, $i\in I_j$, for at least one group $I_j$, $j=1,2,\dots ,m$, are conflicting and/or the resulting group supervisors at the higher level are conflicting. This issue can be solved by computing coordinators for nonblockingness that we have presented in~\cite{JDEDS} for the basic coordination control architecture with a single (centralized) coordinator that can now be qualified as the two-level coordination control architecture.

  It appears then natural to combine the bottom-up computation of a posteriori supervisors with the bottom-up computation of coordinators for nonblockingness, which is proposed in this section. First of all, it should be noted that, unlike the prefix-closed case, we do not have a general distributed procedure to compute the supremal conditionally controllable and normal languages. We have shown in~\cite{JDEDS} that, for the two-level coordination control architecture, the maximally permissive solutions for non-prefix-closed languages can be computed in a similar distributed way if the optimal supervisor for the coordinator is included in the optimal local supervisors projected to the coordinator alphabet: $\supC_k \subseteq P_k(\supC_{i+k})$ for all local supervisors $i$. We recall at this point that the opposite inclusion is always true and if the equality $\supC_k \subseteq P_k(\supC_{i+k})$ does not hold, one may still compute local supervisors $\supC_{i+k}$ as described in~\cite{JDEDS}, but the maximal permissiveness cannot be guaranteed.

  Moreover, the typical issue with non-prefix-closed languages is that the local supervisors $\supCN_{i+k_j}$, $i\in I_j$, after the application of the a posteriori group supervisors $\widetilde{ \supCN_{k_j}}$ are conflicting in general, which corresponds to the blocking case. Let us recall that group supervisors for groups $j=1,2,\dots ,m$ are computed as follows, cf. Eq.~\ref{distr_computation}: $\widetilde{ \supCN_{k_j}} =\ \parallel_{i\in I_j} \supCN(P_{k_j}(\supCN_{i+k_j}), L(G_{k_j})$. We propose to apply Theorem~\ref{thm:nonblocking} to all groups $j=1,2,\dots, m$, where $\supCN_{i+k_j} \parallel \widetilde{  \supCN_{k_j}}$, $i\in I_j$, denoted by $\widetilde{ \supCN_{i+k_j}}$, are blocking. Namely, we have to extend the alphabets $A_{k_j}$ so that the observer conditions of Theorem~\ref{thm:nonblocking} are met. Namely, we need to extend the alphabets $A_{k_j}$ so that $P_{k_j}:(A_{i+k_j})^*\to 
( A{k_j})^*$ be $\widetilde{  \supCN_{i+k_j}}$-observer, for all $i \in I_j$.

  The group coordinators for nonblockingness can now be computed as follows
  \begin{align}
  \label{lowlevel_cb}
    C_{k_j} &= \supCN( \parallel_{i\in I_j} P_{k_j}(\widetilde{  \supCN_{i+k_j}}), \\ 
&\parallel_{i\in I_j} 
\overline{P_{k_j}(\widetilde{  \supCN_{i+k_j}})},A_{k_j,u},Q_{k_j}) \nonumber \,.
  \end{align}
  This means that the final nonblocking supervisor for the group $j\in \{1,\dots ,m\}$ is given by $\parallel_{i\in I_j}\widetilde{  \supCN_{i+k_j}}\parallel C_{k_j}$ and we denote it by $N_j$.

  Similarly as within the low-level groups, it may happen that for $K$ that is not prefix-closed, the languages resulting from the group supervisors $N_j\subseteq A_{I_j}$, $j=1,2,\dots ,m$, are conflicting, thus leading to blocking. Then, Theorem~\ref{thm:nonblocking} can be used again. This means that we extend the high-level coordinator alphabet $A_{k}$  so that the observer conditions of Theorem~\ref{thm:nonblocking} is satisfied. A high-level coordinator for nonblockingness is then defined by
  \begin{equation}
  \label{highlevel_cb}
    C_{k} =\supCN( \parallel_{j=1}^m P_{k}(N_j), \parallel_{j=1}^m \overline{P_{k}(N_j)}, A_{k,u},Q_{k})\,,
  \end{equation}
  where $A_k$ is the extension of the original (for safety) high level coordinator such that $P_k:A_{I_j}^*\to (A_{I_j}\cap A_k)^*$, be $N_j$-observer, for all $j=1,\dots ,m$.

  Now we are ready to formally propose the combined approach consisting in the following top-down design of coordinators followed by the bottom-up computations of a posteriori supervisors and coordinators for nonblockingness.

  The combined approach is formalized in Procedure~\ref{proc:combined} below. The organizations of subsystems into a hierarchical structure with low-level groups is assumed to be given.
  \begin{algorithm}
    \floatname{algorithm}{Procedure}
    \caption{The combined approach}
    \label{proc:combined}
    \begin{enumerate}
      \item Extend the shared alphabet $A_{sh}$ to high-level coordinator alphabet $A_k\supseteq A_{sh}$ such that $K =\ \parallel_{r=1}^m P_{I_r + k} (K)$.  
    
      \item Construct the high-level coordinator $G_k=P_k(\parallel_{r=1}^m L_r^{hi})$ and set $L_{k}=L(G_{k})$. 
    
      \item For all low-level groups $I_j$, $j=1,2,\dots ,m$, extend the shared event sets of groups $A_{sh,j}$ to low-level coordinator alphabets $A_{k_j}\supseteq  A_{sh,j}$ so that $P_{I_j+k} (K) =\ \parallel_{i\in I_j} P_{i+k_j} (K)$.
    
      \item Construct the coordinators for low-level groups, that is, $G_{k_j} = \|_{\ell\in I_j} P_{k_j} (G_\ell)$ and set $L_{k_j}=L(G_{k_j})$.
    
      \item Compute the supervisors $\supCN_{k_j} = \supCN(P_{k_j}(K),$ $L(G_{k_j}),A_{k_j,u}, Q_{k_j})$ for group coordinators $L(G_{k_j})$, $j=1,2,\dots ,m$.

      \item Compute supervisors $\supCN_{i+k_j} =\ \supCN(P_{i+k_j}(K),$ $L(G_i) \parallel \supCN_{k_j}, A_{i+k_j,u}, Q_{i+k_j})$ for subsystems $i\in I_j$ and for all groups $I_j$, $j=1,2,\dots ,m$.

      \item Compute the a posteriori supervisors $\widetilde{  \supCN_{k_j}} =\ \cap_{i\in I_j}\ \supCN(P_{k_j}(\supCN_{i+k_j}), L(G_{k_j}), A_{k_j,u},Q_{k_j})$ for all groups $j=1,2,\dots ,m$.

      \item For all groups $j\in \{1,2,\dots,m \}$ such that 
$ \widetilde{\supCN_{i+k_j}}:= \supCN_{i+k_j} \|\widetilde{ \supCN_{k_j}}$, for $i\in I_j$, are conflicting (cf. Eq.~(\ref{supervisor_resulting_low})), compute the group coordinators for nonblockingness using Eq.~(\ref{lowlevel_cb}), that is, $C_{k_j} = \supCN( \parallel_{i\in I_j} P_{k_j}(\widetilde{  \supCN_{i+k_j}}),  \parallel_{i\in I_j} 
\overline{P_{k_j}(\widetilde{  \supCN_{i+k_j}})},A_{k_j,u},Q_{k_j})  $, and set $C_{k_j}=A_{k_j}^*$ for all groups, where $\widetilde{  \supCN_{i+k_j}}$ are not conflicting. Then the language $N_j =\ \parallel_{i\in I_j}\widetilde{  \supCN_{i+k_j}}\parallel C_{k_j}$ is the resulting nonblocking supervisor for the group $j$.
    
      \item Compute the a posteriori supervisor $\widetilde{  \supCN_k}$ at the high-level (cf. Eq.~\ref{aposteriori_high}).
    
      \item If the languages $N_j\parallel \widetilde{  \supCN_k}$ are conflicting, then compute the high-level coordinator for nonblocking $C_k$ using Eq.~(\ref{highlevel_cb}), i.e. $C_{k} =\supCN( \parallel_{j=1}^m P_{k}(N_j), \parallel_{j=1}^m \overline{P_{k}(N_j)}, A_{k,u},Q_{k})$ and set $C_k=A_k^*$ if the languages $N_j\parallel \widetilde{  \supCN_k}$ are not conflicting.
    
      \item Set $N_j\parallel \widetilde{  \supCN_k} \parallel C_k$ as the final closed-loop of the three-level coordination control based on the combined approach. 
    \end{enumerate}
  \end{algorithm}
  
  We have shown in previous sections that, for prefix-closed languages, Procedure~\ref{proc:combined} yields the supremal three-level conditionally controllable and conditionally normal sublanguage of $K$. This cannot be guaranteed in the general case, however, we have a distributive and hierarchical (sometimes referred to as heterarchical) way to compute a safe (although possibly not maximally permissive) and nonblocking supervisor.

  We note that the computational complexity of all steps in Procedure~\ref{proc:combined} is polynomial in fairly small parameters (number of states and events of subsystems combined with coordinators) provided the projection to all coordinator alphabets satisfy the observer condition, in which case there is no problem with possibly an exponential size of the projected generators, and these are guaranteed to be smaller than the non-projected generators.

\section{Concluding remarks}\label{conclusion}
  We proposed a new general approach to coordination control of DES with partial observations. The approach combines the advantages of both the top-down and the bottom-up approaches proposed earlier. It consists in a top-down computation of coordinators (first a high-level coordinator is computed and then the group coordinators are computed) followed by the computation of supervisors at the lowest level (for individual subsystems) and, finally, the a posteriori supervisors and coordinators for nonblockingness are computed in a bottom-up manner.

  The main advantage of the approach is that it combines the main advantage of the top-down approach---the possibility to compute local supervisors only for the individual subsystems---with the generality offered by the bottom-up approach that has namely enabled to leave out the restrictive conditions for being able to compute maximally permissive solutions in a distributed way and to leave out the nonconflictingness assumptions owing to the bottom-up computation of coordinators for nonblockingness. In the near future, we plan to apply the combined approach to discrete-event models of large scale systems stemming from manufacturing and traffic systems. We recall that recently a weaker condition than normality, called relative observability, was proposed for monolithic partially observed DES, cf.~\cite{caiCDC13}. It is possible to introduce a distributed version of relative observability, conditional relative observability~\cite{KomendaMS14a} and use it in our multilevel architecture instead of normality.

\section{Acknowledgments}
  The authors thank S. Lafortune and F. Lin for a fruitful discussion. The research was supported by RVO 67985840, by the Czech Ministry of Education in project MUSIC (grant LH13012), and by the DFG in project DIAMOND (Emmy Noether grant KR~4381/1-1).

\end{document}